\documentclass[a4paper]{article}
\usepackage[utf8]{inputenc}
\usepackage{hyperref}
\usepackage{amsmath, amsthm}
\usepackage{amsfonts}
\usepackage{amssymb}
\usepackage{mathtools}
\usepackage[dvipsnames]{xcolor}
\hypersetup{
  colorlinks=true,
  citecolor=blue,
  linkcolor=purple!80
}
  
\usepackage{subcaption}
\usepackage{enumitem}
\usepackage{microtype}
\usepackage[nameinlink,noabbrev,capitalize]{cleveref}
\crefname{ass}{Assumption}{Assumptions}

\usepackage{comment}

\usepackage[isbn=false,maxbibnames=99]{biblatex}
\newcommand{\grad}{\nabla}

\newcommand{\norm}[2]{\left\lVert #1\right\rVert_{#2}}
\DeclarePairedDelimiterX\dual[2]{\langle}{\rangle}{#1,#2}
\DeclarePairedDelimiter\parens()
\DeclarePairedDelimiter\abs\lvert\rvert
\DeclarePairedDelimiter\set\{\}
\newcommand{\weaklyto}{\rightharpoonup}

\newcommand{\WOT}{\overset{\mathrm{WOT}}{\to}}
\newcommand{\gammato}{\overset{\gamma}\to}

\newcommand{\M}{\mathcal{M}}
\newcommand{\K}{\mathcal{K}}

\newcommand\kzero{k}
\newcommand\kone{K}
\newcommand\jj{j}

\newcommand{\Qf}{L}

\newcommand\N{\mathbb{N}}
\newcommand\R{\mathbb{R}}

\renewcommand\d{\mathop{}\!\mathrm{d}}

\DeclareMathOperator*{\capa}{cap}

\newtheorem{theorem}{Theorem}[section]
\newtheorem{prop}[theorem]{Proposition}
\newtheorem{lem}[theorem]{Lemma}

\newtheorem{eg}[theorem]{Example}
\newtheorem{remark}[theorem]{Remark}
\newtheorem{ass}[theorem]{Assumption}

\theoremstyle{definition}
\newtheorem{defn}[theorem]{Definition}

\title{Subdifferentials and penalty approximations of the obstacle problem}
\author{Amal Alphonse\thanks{Weierstrass Institute, Mohrenstrasse 39, 10117 Berlin, Germany (\href{mailto:alphonse@wias-berlin.de}{alphonse@wias-berlin.de})} \and Gerd Wachsmuth\thanks{Institute of Mathematics, Brandenburgische Technische Universität Cottbus-Senftenberg, 03046 Cottbus, Germany (\href{mailto:gerd.wachsmuth@b-tu.de}{gerd.wachsmuth@b-tu.de}), ORCID: \href{https://orcid.org/0000-0002-3098-1503}{0000-0002-3098-1503}\protect\\
{\textbf{Acknowledgements:} the authors are grateful to the referees for their useful comments.}}}

\begin{document}
\maketitle
\begin{abstract}
	We consider a framework for approximating the obstacle problem through a penalty approach by nonlinear PDEs.
	By using tools from capacity theory, we show that derivatives of the solution maps of the penalised problems converge in the weak operator topology to an element of the strong-weak Bouligand subdifferential. We are able to treat smooth penalty terms as well as nonsmooth ones involving for example the positive part function $\max(0,\cdot)$. Our abstract framework applies to several specific choices of penalty functions which are omnipresent in the literature. We conclude with consequences to the theory of optimal control of the obstacle problem.
\end{abstract}

 \section{Introduction}
A ubiquitous method to approximate solutions of the classical obstacle problem
\begin{equation}
u \in H^1_0(\Omega), \; u \leq \psi : \langle -\Delta u-f, u-v \rangle \leq 0 \qquad \forall v \in H^1_0(\Omega), \; v \leq \psi \label{eq:VI}
\end{equation}
is by penalisation through a nonlinear PDE
\begin{equation}
-\Delta u_\rho + \Lambda_\rho(u_\rho-\psi) = f_\rho\label{eq:pde}
\end{equation}
{where $\Lambda_\rho$ is a suitable (nonlinear) operator.} The equation \eqref{eq:pde} approximates the variational inequality \eqref{eq:VI}  in the sense that its solutions satisfy $u_\rho \to u$ as $\rho \to 0$ provided $f_\rho \to f$. If we define the source-to-solution map $f_\rho \mapsto u_\rho$ of \eqref{eq:pde} by
\[S_\rho \colon H^{-1}(\Omega) \to H_0^1(\Omega)\]
and consider its derivative at $f_\rho$ in a direction $d$ denoted by $\alpha_\rho := S_\rho'(f_\rho)(d)$, then we know that it satisfies the linearised equation 
\begin{equation}
-\Delta\alpha_\rho + \Lambda_\rho'(u_\rho-\psi)(\alpha_\rho) = d.\label{eq:pdeForDerivativeStart}
\end{equation}
A natural question arises concerning the convergence of the derivatives
$\alpha_\rho = S_\rho'(f_\rho)(d)$ in the limit $\rho \to 0$ {(the derivatives are indeed uniformly bounded as we shall see later in \cref{lem:boundOnDerivative})}, or more generally,
the convergence of the operators $S_\rho'(f_\rho)\colon H^{-1}(\Omega) \to H^1_0(\Omega)$.
The possibility that immediately comes to mind is that they converge to the directional derivative of the VI solution map $S \colon f \mapsto u$ in \eqref{eq:VI}, i.e., to $S'(f)$. If this were to be the case, the limit $\alpha(d)$ would satisfy 
\[\alpha(d) \in \K : \langle -\Delta\alpha(d)-d, \alpha(d)-v \rangle \leq 0 \quad \forall v \in \K
\]
where $\K := \{ v \in H^1_0(\Omega) : \text{$v\leq 0$ q.e. on $\{u=\psi\}$, $\langle -\Delta u-f, v \rangle =0$}\}$ is the critical cone at $u=S(f)$ {and the notation q.e. means quasi-everywhere, see \cref{sec:orthogonality_lap_alpha} for more details}. 
In general,
this cannot happen
since $d \mapsto \alpha(d)$ is nonlinear (unless $\K$ is a subspace),
whereas $d \mapsto \alpha_\rho(d)$ is often linear for all $\rho$ (e.g.,
if $\Lambda_\rho$ is Fréchet differentiable) and linearity would be preserved in the limit. 
{We will illustrate this with a concrete example below in \cref{eg:first}.
}

In this paper, we consider limits of the operator $S_\rho'(f_\rho) \colon d \mapsto \alpha_\rho$
for different choices of the penalty function $\Lambda_\rho$ and prove that (under weak conditions) they converge in the weak operator topology to elements of the strong-weak subdifferential of $S$, defined as
\begin{align*}
	\partial_B^{sw} S(f) := \set*{
		\begin{aligned}
			L \in \mathcal{L}(H^{-1}(\Omega), H^1_0(\Omega)) : \;
			& \exists \{f_n\} \subset F_S : f_n \to f \in H^{-1}(\Omega),
			\\
			& S'(f_n) \WOT L \text{ in $\mathcal{L}(H^{-1}(\Omega), H^1_0(\Omega))$}
		\end{aligned}
	},
\end{align*}
where $F_S$ is the set of all points in $H^{-1}(\Omega)$ at which $S$ is G\^ateaux differentiable and $\WOT$ refers to convergence in the weak operator topology, which we recall now.
\begin{defn}\label{defn:WOT}
	A sequence $\{L_n\} \subset \mathcal{L}(X,Y)$ of bounded linear operators between Banach spaces $X$ and $Y$ converges to a bounded linear operator $L \in \mathcal{L}(X,Y)$ in the weak operator topology if and only if $L_nx \weaklyto Lx$ in $Y$ for all $x \in  X$. We write this as $L_n \WOT L$.
\end{defn}
{Before we proceed, let us give an example demonstrating an element of this subdifferential and verify that it differs from the directional derivative.
\begin{eg}\label{eg:first}
Consider the setting $\psi\equiv 0$. Clearly $S(0)=0$, and using the positive homogeneity of $S$, the perturbation $S(td)=tS(d)$ for all $d \in H^1_0(\Omega)$ and $t>0$. Hence the directional derivative $S'(0)(d)$ equals $S(d)$, i.e., 
\[S'(0) = S\]
is itself the solution map of the obstacle problem. On the other hand, consider the penalised equation \eqref{eq:pde} with $\Lambda_\rho$ chosen as a $C^1$ regulariser satisfying $\Lambda_\rho(0)=\Lambda_\rho'(0)=0$ (the first choice in \eqref{eq:examples_smooth} satisfies this, see \cref{sec:abstract}). We find that $S_\rho(0) = 0$ because $\Lambda_\rho(0)=0$, and, recalling \eqref{eq:pdeForDerivativeStart}, the derivative $\alpha_\rho(d) :=S'(0)(d)$ satisfies the equation $-\Delta \alpha_\rho(d) = d$ since $\Lambda_\rho'(0)=0$. That is, $S_\rho'(0)=(-\Delta)^{-1}$ is the inverse Laplacian for all $\rho>0$. Thus the limiting object remains the same:
\[\lim_{\rho \to 0} S_\rho'(0) = (-\Delta)^{-1}.\]
This of course does not agree with $S'(0)=S$.
Our main result \cref{thm:main_result} will show
that $(-\Delta)^{-1} \in \partial_B^{sw}S(0)$.
\end{eg}}
Our method of proof relies on a recent characterisation of $\partial^{sw}_B S(f)$ from \cite{RW} (see \eqref{eq:RW_characterisation} below) involving so-called capacitary measures,
which we will introduce in \cref{sec:char_limits}.
In order to fulfil one of the conditions to utilise that characterisation (see also \cref{rem:cts_solution}), we assume the following.
\begin{ass}[Standing assumption on regularity]\label{ass:for_Bsw}
	Let $\Omega \subset \R^n$ be a bounded open set in dimension $n \ge 2$.
	For the obstacle, we assume
	$\psi \in C(\bar\Omega) \cap H^1(\Omega)$ with either $\psi \in H^1_0(\Omega)$ or $\psi > 0$ on $\partial\Omega$.

\end{ass}
For convenience,
	we define the positive part $(\cdot)^+$ and the negative part $(\cdot)^-$ via
	\begin{equation*}
		u^+ := \max(0, u)
		\qquad\text{and}\qquad
		u^- := \max(0,-u)
	\end{equation*}
	respectively.
Inspired by tradition and existing literature, we treat in particular the following specific examples\footnote{We have chosen the superscripts $\mathrm{m}$ and $\mathrm{c}$ in \eqref{eq:examples_ns} to stand for \textbf{m}ax and \textbf{c}omplementarity respectively; the reason for the former is clear and the latter is due to the fact that $\Lambda^{\mathrm{c}}_\rho$ is obtained from writing the VI as a complementarity system. The superscripts $\mathrm{sm}$ and $\mathrm{sc}$ in \eqref{eq:examples_smooth} are supposed to denote that these are \textbf{s}moothed versions of $\mathrm{m}$ and $\mathrm{c}$ respectively.}:
\begin{align}
 \Lambda_\rho^\mathrm{m}(u) = \frac 1\rho u^+\qquad &\text{and}\qquad\Lambda_\rho^\mathrm{c}(u) =  \frac{1}{\rho}(\rho\bar\lambda + u)^+ \label{eq:examples_ns}
\end{align}
where $\bar\lambda \in L^\infty(\Omega)$, $\bar\lambda \geq 0$ is given, 
and the corresponding smooth versions
\begin{align}
 \Lambda_\rho^\mathrm{sm}(u) = \frac 1\rho m_\rho(u) \qquad &\text{and}\qquad  \Lambda_\rho^\mathrm{sc}(u) = \frac{1}{\rho}m_\rho (\rho\bar\lambda +  u)\label{eq:examples_smooth}
 \end{align}
where $m_\rho$ is a regularisation (satisfying \cref{ass:0}) that smooths out the positive part function $(\cdot)^+$, see \cref{lem:regularisations_satisfy_ass} for some concrete examples. In particular, we have in mind the commonly used regularisations from \cite{MR2822818} and \cite{KunischWachsmuthPathFollowing} respectively, see \eqref{eq:eg_mrho_new} and \eqref{eq:eg_mgamma_new} for their definitions. We will denote solution maps as well as other maps that depend on the specific choice of $\Lambda$ with the superscript $\mathrm{m}, \mathrm{c}, \mathrm{sm}$ or $\mathrm{sc}$ as appropriate. Note that the choice $\bar\lambda \equiv 0$ yields $\Lambda_\rho^\mathrm{m} = \Lambda_\rho^\mathrm{c}$ and $\Lambda_\rho^\mathrm{sm} = \Lambda_\rho^\mathrm{sc}$. A typical choice  of $\bar\lambda$
is $\bar\lambda := (f+\Delta\psi)^+$, see \cite[Theorem 3.2]{IKOCVIs}.

Our penalty term $\Lambda_\rho^{\mathrm{sc}}$ covers also the penalisation
\begin{equation}
\Lambda_\rho^{\widetilde{\mathrm{sc}}}(u) = \tilde m_\rho(\bar\lambda + (1\slash \rho) u),\label{eq:tildeLambda_rho}
\end{equation}
(given a regularisation $\tilde m_\rho$) which is used frequently in the literature, see, e.g., \cite{KunischWachsmuth, KunischWachsmuthPathFollowing, SchielaWachsmuth, IK, IKOCVIs}.
Indeed,
$\Lambda_\rho^{\widetilde{\mathrm{sc}}}$ is of the form \eqref{eq:examples_smooth} using
\begin{equation}
	\label{eq:machet_die_tilde_weg}
	m_\rho(r) := \rho \tilde m_\rho( r / \rho ),
\end{equation}
see \cref{rem:tilde_penalisation} for more details.

Let us give the main results of this work. 
We start with the smooth case \eqref{eq:examples_smooth}. 
Below, we use the notation $C_0(\Omega)$ for the set of functions $v \in C(\bar\Omega)$ with $v = 0$ on $\partial\Omega$.

\begin{theorem}\label{thm:main_result_intro}
Let \cref{ass:for_Bsw} hold. For every $f \in H^{-1}(\Omega)$ with $S(f) \in C_0(\Omega)$, if  $f_\rho \to f$ in $H^{-1}(\Omega)$, then there exist maps
$\Qf^\mathrm{sm}, \Qf^\mathrm{sc} \in \partial_B^{sw}S(f)$ 
such that for a subsequence (that we relabel),
\begin{equation*}
(S_\rho^{\mathrm{sm}})'(f_\rho) \WOT \Qf^\mathrm{sm},
\qquad\text{and}\qquad
(S_\rho^{\mathrm{sc}})'(f_\rho) \WOT \Qf^\mathrm{sc}. 
\end{equation*}
\end{theorem}
\cref{thm:main_result_intro}
is a special case of \cref{thm:main_result}
below.

\begin{remark}[On the assumption $S(f) \in C_0(\Omega)$]\label{rem:cts_solution}
	Note that asking for the solution $u = S(f)$ of \eqref{eq:VI} to satisfy $u \in C_0(\Omega)$ (which is needed for the characterisation of \cite[Lemma~4.3]{RW}) is not too restrictive. For example, when $\Omega$ is Lipschitz,   \cite[Theorem 2.7, \S 5]{Rodrigues} guarantees $u \in C(\bar\Omega) \cap H^1_0(\Omega)$ if $\psi \in C(\bar\Omega)$ and $f \in W^{-1,p}(\Omega)$ for $p>n$. An alternative is if $\Omega$ satisfies a uniform exterior cone condition and $n \leq 3$, $\psi \in H^1(\Omega)$ with $\Delta\psi \in L^2(\Omega)$ and $f \in L^2(\Omega)$: then $u$ is even Hölder continuous, see \cite[Theorem 2.7]{NeitzelWachsmuth}.

\end{remark}

In the nonsmooth case \eqref{eq:examples_ns},
we additionally have to assume Gâteaux differentiability of the approximations.
\begin{theorem}\label{thm:intro_nonsmooth}
Let \cref{ass:for_Bsw} hold. For every sequence $\{f_\rho\} \subset H^{-1}(\Omega)$ such that $S_\rho^{\mathrm{m}}$ ($S_\rho^{\mathrm{c}}$) is
	Gâteaux differentiable at $f_\rho$
	and $f_\rho \to f$ with $S(f) \in C_0(\Omega)$, 
there exists a map $\Qf^{\mathrm{m}} \in \partial_B^{sw}S(f)$ ($\Qf^{\mathrm{c}} \in \partial_B^{sw}S(f)$)  such that for a subsequence (that we relabel),
\begin{equation*}
(S_\rho^{\mathrm{m}})'(f_\rho) \WOT \Qf^{\mathrm{m}},
\qquad
((S_\rho^{\mathrm{c}})'(f_\rho) \WOT \Qf^{\mathrm{c}}).
\end{equation*}
\end{theorem}
\cref{thm:intro_nonsmooth}
is also a special case of \cref{thm:main_result}
below.

The nonsmooth result \cref{thm:intro_nonsmooth} contains the assumption that each $S_\rho$ is G\^ateaux differentiable at $f_\rho$, which can be justified in the following sense. We rely on the key observation that there exists a dense set $F \subset H^{-1}(\Omega)$ such that $S_\rho^{\mathrm{m}}$ and $S_\rho^{\mathrm{c}}$ are G\^ateaux differentiable from $F$ into $H^1_0(\Omega)$ for every $\rho$ taken from a countable set (see \cref{lem:density_result_general}).
Then the above theorem can always be applied for every constant sequence $f_{\rho_n} \equiv f \in F$ such that $S(f) \in C_0(\Omega)$.

In order to cover these cases without repetition and to generalise the structure of the penalty term as much as possible, we consider an abstract problem formulation, as described in the next subsection. In fact, as mentioned, \cref{thm:main_result_intro} and \cref{thm:intro_nonsmooth} are consequences of our more general result \cref{thm:main_result}. Finally, we also mention \cref{thm:oc_result} where we obtain a first-order stationarity condition for an optimal control problem with a VI constraint.

\begin{remark}[Generalisation to other elliptic operators]\label{rem:on_A}
In this paper we consider the elliptic operator in \eqref{eq:VI} and related problems to be the Laplacian because the characterisation of $\partial_B^{sw}S(f)$ from \cite{RW} was shown for the Laplacian.  Our results up to (but not including) the proof of \cref{thm:main_result} work in a more general setting where $-\Delta$ is replaced by a linear, bounded, coercive operator
$A\colon H^1_0(\Omega) \to H^{-1}(\Omega)$ {which additionally satisfies $\langle Au^-, u^+ \rangle \leq 0$. 
}
The estimates below should be adjusted to include the coercivity and boundedness constants. 
\end{remark}
\section{Abstract setup and properties of the penalised problem}\label{sec:abstract}
Throughout,
we equip the Sobolev space $H_0^1(\Omega)$ with the inner product
\begin{equation*}
	(u, v)_{H_0^1(\Omega)} := \int_\Omega \nabla u \cdot \nabla v \d x.
\end{equation*}

\subsection{Setup}

For the results of this section, it suffices to take an obstacle $\psi \in H^1(\Omega)$ with $\psi|_{\partial\Omega} \geq 0$ in the sense that $\min(0, \psi) \in H_0^1(\Omega)$.

Let us now formulate an abstract penalty term. For each $\rho>0$, we work with a general mapping $\Lambda_\rho\colon H^1_0(\Omega) \to H^{-1}(\Omega)$  defined as a Nemytskii map of a function $\lambda_\rho \colon \Omega \times \mathbb{R} \to \mathbb{R}$,
i.e.,
\[\Lambda_\rho(u)(x) := \lambda_\rho(x, u(x)).\]
We make the following standing assumption on $\lambda_\rho$.
\begin{ass}\label{ass:1}
We assume that 
\begin{enumerate}[label=(\roman*)]\itemsep=0em
\item\label{item:i} for all $\rho \in (0,\infty)$,
	$\lambda_\rho\colon \Omega \times \mathbb{R} \to \mathbb{R}$ is a Carathéodory function,
\item\label{item:ii} for all $\rho \in (0,\infty)$,
	$\lambda_\rho(x, \cdot) \colon \mathbb{R} \to \mathbb{R}$ is increasing and convex for a.a.\ $x \in \Omega$,
\item\label{item:iii} for all $\rho \in (0,\infty)$,
	there exist
	$\kzero_\rho, \kone_\rho \in L^\infty(\Omega)$
	and
	$\jj_\rho \in L^2(\Omega)$
	with
	$\kzero_\rho \leq \kone_\rho$
	such that
	\begin{equation}
		\label{eq:structure_lambda}
		\lambda_\rho(x, r) = \begin{cases}
			0 &\text{if } r \leq \kzero_\rho(x),\\
			\frac{r+\jj_\rho(x)}{\rho} &\text{if } r  \geq \kone_\rho(x),
		\end{cases}
	\end{equation}
and, for some $C > 0$,
\begin{align}
\kzero_\rho &\to 0 \text{ in } L^\infty(\Omega)\text{ as } \rho \to 0\label{ass:k0ConvergesToZeroLinf_purple},\\
\norm{\kzero_\rho}{L^2(\Omega)} +
\norm{\kone_\rho}{L^2(\Omega)}  &\leq C\rho \quad\forall\rho \in (0,1]\label{ass:kgrowth},
\end{align}
\item\label{item:iv} $\lambda_\rho(x, \cdot) \in C^1(\mathbb{R})$ for all $\rho \in (0,\infty)$ and a.a.\ $x \in \Omega$
or $\kzero_\rho \equiv \kone_\rho$. 
\end{enumerate}
\end{ass}
Throughout this work, bearing in mind \cref{ass:1} \ref{item:iv}, we refer to the case that $\lambda_\rho(x,\cdot) \in C^1(\R)$ for a.a.\ $x \in \Omega$
			as the ``smooth case''
			and to $\kzero_\rho \equiv \kone_\rho$
			as the ``nonsmooth case''. 
\begin{remark}\leavevmode
	\begin{enumerate}[label=(\roman*)]\itemsep=0em
		\item
			The smoothness $\lambda_\rho(x,\cdot) \in C^1(\R)$ is equivalent to $\kzero_\rho(x) < \kone_\rho(x)$.
		\item Since we know that $\lambda_\rho(x,\cdot) \geq 0$, we have $\lambda_\rho(x,\kone_\rho(x)) \geq 0$, which yields
			\begin{equation}
				\kone_\rho+\jj_\rho \geq 0\label{ass:k1PlusJNonNeg}.
			\end{equation}
	\end{enumerate}
\end{remark}

\begin{lem}
	\label{lem:satisfaction_new}
\cref{ass:1} is satisfied for the nonsmooth choices in \eqref{eq:examples_ns}.
\end{lem}
\begin{proof}
Observe that for $\lambda_\rho^\mathrm{m}(r) = \frac{1}{\rho}r^+$,  
\[\kzero_\rho=\kone_\rho=\jj_\rho=0,\]
whereas for $\lambda_\rho^\mathrm{c} (x,r) = (1\slash \rho)(\rho\bar \lambda(x) +  r)^+$, we have
\[\kzero_\rho(x) = \kone_\rho(x) = - \rho\bar \lambda(x), \quad \jj_\rho(x) = \rho\bar\lambda(x).\]
From this information we can verify the claim without difficulty.
\end{proof}
Regarding the smooth cases, let us first give sufficient conditions on the structure of $m_\rho$ that will enable us to verify \cref{ass:1} more easily.
\begin{ass}\label{ass:0}
	Let  $m_\rho \in C^1(\R)$ be given for all $\rho > 0$,
	such that there exist
	$\theta_\rho, \Theta_\rho, l_\rho \in \mathbb{R}$  with
	\begin{equation*}
		m_\rho(r) = 0 \quad\text{for all } r \le \theta_\rho,
		\qquad
		m_\rho(r) = r + l_\rho \quad\text{for all } r \ge \Theta_\rho,
	\end{equation*}
	for all $\rho > 0$
	and
	\begin{equation}
		|\theta_\rho| +
		|\Theta_\rho| \leq C\rho \quad\forall\rho \in (0,1]\label{ass:Thetagrowth}
	\end{equation}
	for some constant $C > 0$.
\end{ass}

\begin{lem}
If a function $m_\rho$ satisfies \cref{ass:0}, then \cref{ass:1} is satisfied by both $\Lambda_\rho^{\mathrm{sm}}$ and $\Lambda_\rho^{\mathrm{sc}}$ defined as in \eqref{eq:examples_smooth}.
\end{lem}
\begin{proof}
The first part of \cref{ass:0} directly implies that \cref{ass:1} \ref{item:i}, \ref{item:ii} and \ref{item:iv} are satisfied. The remaining part follows from observing that in the $\Lambda_\rho^{\mathrm{sm}}$ case, we have $\kzero_\rho = \theta_\rho$, $\kone_\rho = \Theta_\rho$ and $j_\rho = l_\rho$,   and in the $\Lambda_\rho^{\mathrm{sc}}$ case, we have $\kzero_\rho = \theta_\rho-\rho\bar\lambda$, $K_\rho = \Theta_\rho-\rho\bar\lambda$ and $j_\rho = \rho\bar\lambda + l_\rho$. 

\end{proof}
\begin{remark}\label{rem:tilde_penalisation}
The penalisation $\Lambda_\rho^{\widetilde{\mathrm{sc}}}$ from \eqref{eq:tildeLambda_rho} can be handled too and in fact we can in this situation weaken the last condition of \cref{ass:0}: we would need the parameters of $\tilde m_\rho$ to satisfy only
\begin{equation*}
|\theta_\rho| +
|\Theta_\rho| \leq C\quad\forall\rho \in (0,1].
\end{equation*}
Indeed, if $\tilde m_\rho$ is a regularisation with the structure presented in \cref{ass:0}
we find that $m_\rho$ defined as in \eqref{eq:machet_die_tilde_weg} has associated parameters $\hat\theta_\rho := \rho\theta_\rho$, $\hat\Theta_\rho := \rho\Theta_\rho$ and $\hat l_\rho := \rho l_\rho$, which all contain a helpful factor of $\rho$.
\end{remark} 
Let us now look at some examples of $m_\rho$ that satisfy \cref{ass:0}.
\begin{lem}\label{lem:regularisations_satisfy_ass}
The following choices of $m_\rho$ satisfy \cref{ass:0}.
\begin{enumerate}[label=(\roman*)]
\item The global regularisation {(also known as the Huber regularisation)} used, e.g., in \cite{MR2822818}:
\begin{align}\label{eq:eg_mrho_new}
m_\rho(r)
:= \begin{cases}
0 &\text{if } r \leq 0,\\
\frac{r^2}{2\rho} &\text{if } 0 < r < \rho, \\
r-\frac{\rho}{2} &\text{if } r \geq \rho,
\end{cases}
\end{align}
with
$\theta_\rho = 0$, $\Theta_\rho = \rho$, $l_\rho = -\rho\slash 2$. 
\item The regularisation from \cite{KunischWachsmuthPathFollowing}:
\begin{align}\label{eq:eg_mgamma_new}
\tilde{m}_\rho(r) := \begin{cases}
	0 &\text{if } r \leq -\frac{\rho}{2},\\
\frac{1}{2\rho^3}\left(r + \frac{\rho}{2}\right)^3\left(\frac{3\rho}{2}-r\right)  &\text{if $-\frac{\rho}{2} < r < \frac{\rho}{2}$},\\
r &\text{if $r \geq \frac{\rho}{2}$},
\end{cases}
\end{align}
with
$\theta_\rho = -\frac{\rho}{2}$, $\Theta_\rho = \frac{\rho}{2}$, $l_\rho = 0$.
\item The local regularisation from \cite{MR2822818}:
\[m_\rho^l(r)
:= \begin{cases}
0 &\text{if } r \leq -\rho,\\
\frac{r^2}{4\rho} + \frac{r}{2} + \frac{\rho}{4}&\text{if } -\rho < r < \rho, \\
r &\text{if } r \geq \rho,
\end{cases}
\]
with
$\theta_\rho = -\rho$, $\Theta_\rho = \rho$, $l_\rho = 0$.

\item The regularisation from \cite{KunischWachsmuth}: 
\[\hat{m}_\rho(r) := \begin{cases}
		0 & \text{if } r \leq -\frac{\rho}{2},\\
		\frac{1}{2\rho}\left(r + \frac{\rho}{2}\right)^2 & \text{if $-\frac{\rho}{2} < r < \frac{\rho}{2}$},\\
		r & \text{if $r \geq \frac{\rho}{2}$},
\end{cases}\]
with
$\theta_\rho = -\frac{\rho}{2}$, $\Theta_\rho = \frac{\rho}{2}$, $l_\rho = 0$.
\end{enumerate}
Hence $\Lambda_\rho^{sm}$ and $\Lambda_\rho^{sc}$ as defined in \eqref{eq:examples_smooth} associated to each of the above regularisations satisfy \cref{ass:1}.
\end{lem}
\begin{proof}
We can verify the claim using the information presented below each choice.
\end{proof}
 
\subsection{First properties}
\begin{lem}\label{lem:basic_properties}
{Let \cref{ass:1} hold.} 
		For all $\rho \in (0,\infty)$
		and a.a.\ $x \in \Omega$
the following holds.
\begin{enumerate}[label=(\roman*)]
	\item
		The map $\lambda_\rho(x,\cdot)$ is directionally differentiable.
	In the nonsmooth case (with $\kzero_\rho \equiv \kone_\rho$), the directional derivative is
	given by
\begin{equation}
	\lambda_\rho'(x,r)(h) = \frac{1}{\rho}\chi_{\{r = \kzero_\rho\}}(x)h^+ + \frac{1}{\rho}\chi_{\{r > \kzero_\rho\}}(x)h
	\qquad\forall r,h \in \R.
\label{eq:formula_for_dir_der}
\end{equation}
\item In the smooth case we have
\begin{equation}
0 \leq \lambda_\rho'(x,r) \leq \frac{1}{\rho} \qquad \forall r \in \R.\label{eq:lambdaRhoDerivIncIfC1}
\end{equation}
\item We have
\begin{align}
\lambda_\rho'(x, \kzero_\rho(x))(\alpha)\alpha \geq 0,
\quad
\lambda_\rho'(x, \kzero_\rho(x))(\alpha)\alpha^+ &\geq 0 && \forall \alpha \in \mathbb{R},\label{ass:lambdaAtK0Positive}\\
|\lambda_\rho'(x, r)(h)| &\leq \frac{1}{\rho}|h| && \forall r,h \in \R.\label{ass:lambda_rho_deriv_bdd}
\end{align}
\item The function $\lambda_\rho(x,\cdot)$ is Lipschitz continuous uniformly in $x$, i.e.,
\begin{equation}
	|\lambda_\rho(x,u) - \lambda_\rho(x,v)| \leq \frac{1}{\rho} |u-v| \qquad\forall u,v \in \R.\label{eq:lambda_is_lipschitz}
\end{equation}
\item If $\rho \in (0,1]$, we have the growth condition (with the constant $C$ from \eqref{ass:kgrowth})  
\begin{equation}
\norm{\jj_\rho}{L^2(\Omega)} \leq 2 C\rho.\label{eq:jjgrowth}
\end{equation}
\item The map $\Lambda_\rho \colon L^2(\Omega) \to L^2(\Omega)$ is well defined, Lipschitz continuous and directionally differentiable with the derivative
given by 
\begin{equation}
	\label{eq:Lambda_and_lambda}
	\Lambda_\rho'(u)(h)(x) = \lambda_\rho'(x,u(x))(h(x)).
\end{equation}
\item  The map $\Lambda_\rho\colon L^2(\Omega) \to L^2(\Omega)$ is monotone.
\item
	For all $r \in \R$
we have
\begin{equation}
    \lambda_\rho(x,r)r^+ \ge -\frac{K_\rho(x)}{\rho}r^+ + \frac{1}{\rho}|r^+|^2.\label{eq:new_consequence}
\end{equation}
\item
	If $\rho \in (0,1]$,
	we have
	(with the constant $C$ from \eqref{ass:kgrowth})
	\begin{equation}
		\label{eq:NiceEstimateForLambda}
		\norm{\Lambda_\rho(v - \psi)}{L^2(\Omega)}
		\le
		\frac{1}{\rho} \norm{(v - \psi)^+}{L^2(\Omega)} + C
		\qquad\forall v \in H_0^1(\Omega)
		.
	\end{equation}
\end{enumerate}
\end{lem}
\begin{proof}
\begin{enumerate}[label=(\roman*)]
\item
	In the nonsmooth case,
	we have $\kzero_\rho\equiv \kone_\rho$.
	Consequently, the formula for $\lambda_\rho(x,\cdot)$ yields the differentiability everywhere on points other than at $\kzero_\rho$,    {where it is directionally differentiable. Indeed, a direct calculation reveals, for $t>0$, $r, h \in \mathbb{R}$ and fixed $x \in \Omega$,
 \begin{align*}
        \frac{\lambda_\rho(r+th,x)-\lambda_\rho(r,x)}{t} = \begin{cases}
            0 &\text{if } r < \kzero_\rho,\\
            \frac{1}{\rho}h\chi_{\{h \geq 0\}}&\text{if } r=\kzero_\rho,\\
            \frac{1}{\rho}h&\text{if } r > \kzero_\rho,
        \end{cases}
    \end{align*}
    where we used that $\kzero_\rho=\kone_\rho=-j_\rho$ by continuity. From here the expression \eqref{eq:formula_for_dir_der} follows.}
\item The fact that $\lambda_\rho(x,\cdot)$ is convex and differentiable implies that its derivative is increasing, so this follows from the structure of $\lambda_\rho$ given in \cref{ass:1} \ref{item:iii}.
\item 
Regarding \eqref{ass:lambdaAtK0Positive}, for the nonnegativity, this is clear in the smooth case due to the linearity with respect to the direction and the nonnegativity of the derivative. In the nonsmooth case this follows from the expression \eqref{eq:formula_for_dir_der} for the derivative above.

For the upper bound \eqref{ass:lambda_rho_deriv_bdd}, in the smooth case this follows again by \eqref{eq:lambdaRhoDerivIncIfC1}.
In the nonsmooth case this follows by the expression for the directional derivative above and the fact that the sets appearing in the expression are disjoint.

\item Using the mean value theorem for directional derivatives \cite[Proposition 2.29]{MR2986672} and \eqref{ass:lambda_rho_deriv_bdd}, we obtain the claim.
\item We note that by \eqref{eq:lambda_is_lipschitz},
\[\frac1\rho|\kone_\rho + \jj_\rho| =|\lambda_\rho(\cdot, \kone_\rho) - \lambda_\rho(\cdot, \kzero_\rho)| \leq \frac 1\rho |\kone_\rho - \kzero_\rho|\]
which implies $|\jj_\rho| \leq |\kone_\rho-\kzero_\rho| + {|\kone_\rho|}$ whence \eqref{eq:jjgrowth} follows by \eqref{ass:kgrowth}.
\item The above Lipschitz property implies
\begin{align}
|\lambda_\rho(x,u)| = |\lambda_\rho(x,u) - \lambda_\rho(x, \kzero_\rho(x))| \leq \frac{1}{\rho}|u-\kzero_\rho(x)|\label{eq:bound_on_lambda_u}
\end{align} 
and thus $\Lambda_\rho$ maps $L^2(\Omega)$ to $L^2(\Omega)$. Lipschitz continuity also follows easily by the above.

Due to the directional differentiability of $\lambda_\rho(x,\cdot)$ and a simple dominated convergence theorem argument, using the fact that the first derivative is bounded by \eqref{ass:lambda_rho_deriv_bdd}, we obtain directional differentiability. 
\item This follows from the fact that  $\lambda_\rho(x, \cdot)$ is increasing.
\item
From the basic properties of $\lambda_\rho$ and \eqref{ass:k1PlusJNonNeg} we get
\[\lambda_\rho(x, r) \ge \frac{r - K_\rho(x)}{\rho}.\]
Indeed, for $r < K_\rho$, the right-hand side is negative and for $r \ge K_\rho$ this follows from
\eqref{eq:structure_lambda}
and \eqref{ass:k1PlusJNonNeg}.
Now, \eqref{eq:new_consequence} easily follows.
\item
Using the monotonicity of $\lambda_\rho(x,\cdot)$,
$\lambda_\rho(x, \kzero_\rho(x)) = 0$
and the Lipschitz estimate \eqref{eq:lambda_is_lipschitz},
we get
\begin{equation*}
\lambda_\rho(x, v(x)-\psi(x))
\le \lambda_\rho(x, (v(x)-\psi(x))^+)
\leq \frac{1}{\rho}(v(x)-\psi(x))^+ +  \frac 1\rho| \kzero_\rho(x)|
.
\end{equation*}
Taking the $L^2(\Omega)$-norm and using \eqref{ass:kgrowth},
the inequality \eqref{eq:NiceEstimateForLambda} follows.
\qedhere
\end{enumerate}
\end{proof}

Next, we address the well posedness of the PDE \eqref{eq:pde} satisfied by $u_\rho$.
\begin{lem}\label{lem:existence_lipschitz_urho}
{Let \cref{ass:for_Bsw,ass:1} hold.} 	For all $\rho > 0$,
	the solution map $S_\rho\colon H^{-1}(\Omega) \to H^1_0(\Omega)$ of the PDE \eqref{eq:pde} is well defined.
	If $\rho \in (0,1]$,
	we have
\begin{align*}
\norm{S_\rho(f_\rho)}{H_0^1(\Omega)}
		&
		\le
		\norm{f_\rho}{H^{-1}(\Omega)}
		+
		2\norm{\min(0,\psi)}{H_0^1(\Omega)}
		+
		C_P C
		\qquad\forall f_\rho \in H^{-1}(\Omega),
\end{align*}
where $C_P$ is the constant from Poincaré's inequality
and $C$ is from \eqref{ass:kgrowth}.
Furthermore, $S_\rho\colon H^{-1}(\Omega) \to H^1_0(\Omega)$ is Lipschitz continuous with constant $1$.
\end{lem}
\begin{proof}
	We define the operator $T_\rho \colon H_0^1(\Omega) \to H_0^1(\Omega)$
	via
	\begin{equation*}
		T_\rho(u) := (-\Delta)^{-1} \Lambda_\rho(u-\psi)
		\qquad\forall u \in H_0^1(\Omega).
	\end{equation*}
	By applying $(-\Delta)^{-1} \colon H^{-1}(\Omega) \to H_0^1(\Omega)$
	to \eqref{eq:pde}, we get
	\begin{equation*}
		\parens{ \operatorname{Id} + T_\rho } (u_\rho) = (-\Delta)^{-1} f_\rho
		.
	\end{equation*}
	Now, it is easy to check that the operator $T_\rho$
	is continuous and monotone.
	Consequently, the operator $T_\rho$ is maximally monotone, see \cite[Proposition~20.27]{BauschkeCombettes2011}.
	Thus, we can apply Minty's theorem, see \cite[Theorem~21.1 and Proposition~23.8]{BauschkeCombettes2011},
	to obtain that $\operatorname{Id} + T_\rho$
	bijective and the inverse has Lipschitz constant $1$.
	Since $-\Delta \colon H_0^1(\Omega) \to H^{-1}(\Omega)$
	is an isometric isomorphism,
	this shows that $S_\rho$ is well defined and has Lipschitz constant $1$.

	It remains
	to provide an estimate for $S_\rho(f_\rho)$.
	We define $\psi_0 := \min(0,\psi)$ and recall that $\psi_0 \in H_0^1(\Omega)$
	by our assumption on $\psi$.
	We observe $S_\rho(g_\rho) = \psi_0$
	for
	$ g_\rho := -\Delta \psi_0 + \Lambda_\rho(\psi_0 - \psi) $.
	Consequently, {the Lipschitz property} of $S_\rho$ yields
	\begin{align*}
		\norm{S_\rho(f_\rho)}{H_0^1(\Omega)}
		&\le
		\norm{S_\rho(f_\rho) - S_\rho(g_\rho)}{H_0^1(\Omega)} + \norm{S_\rho(g_\rho)}{H_0^1(\Omega)}\\
		&\le
		\norm{f_\rho + \Delta\psi_0 - \Lambda_\rho( \psi_0 - \psi)}{H^{-1}(\Omega)} + \norm{\psi_0}{H^1_0(\Omega)}
		\\&
		\le
		\norm{f_\rho}{H^{-1}(\Omega)}
		+
		2\norm{\psi_0}{H_0^1(\Omega)}
		+
		C_P \norm{\Lambda_\rho(\psi_0 - \psi)}{L^2(\Omega)}		
	\end{align*}
    {with the last inequality following from the triangle inequality,  $\norm{\Delta \psi_0}{H^{-1}(\Omega)} = \norm{\psi_0}{H^1_0(\Omega)}$, and applying Poincaré's inequality to estimate the dual norm of $\Lambda_\rho( \psi_0 - \psi)$.}
	The last addend is estimated via \eqref{eq:NiceEstimateForLambda}
	by using $(\psi_0 - \psi)^+ = 0$.
\end{proof}
Let us now address the differentiability of $S_\rho$.
\begin{lem}\label{lem:boundOnDerivative}
{Let \cref{ass:for_Bsw,ass:1} hold.}
	Given $f_\rho, d \in H^{-1}(\Omega)$, the directional derivative $\alpha_\rho = S_\rho'(f_\rho)(d)$
	exists and
	is the unique solution of the PDE 
\eqref{eq:pdeForDerivativeStart}, i.e.,
\[-\Delta\alpha_\rho + \Lambda_\rho'(u_\rho-\psi)(\alpha_\rho) = d\]
 and satisfies the bound
\[\norm{\alpha_\rho}{H^1_0(\Omega)} \leq  \norm{d}{H^{-1}(\Omega)}.\]
\end{lem}
\begin{proof}
The directional differentiability and
satisfaction of \eqref{eq:pdeForDerivativeStart} follows by the directional differentiability and monotonicity of $\Lambda_\rho$, see e.g. \cite[Lemma 6.1]{AHRW} or \cite[Theorem 2.2]{CCMW}.

For the estimate, we
test \eqref{eq:pdeForDerivativeStart} with $\alpha_\rho$  
and obtain
\[
\norm{\alpha_\rho}{H^1_0(\Omega)}^2 + \int_\Omega \lambda_\rho'(\cdot, u_\rho-\psi)(\alpha_\rho)\alpha_\rho \d x
= \langle d, \alpha_\rho\rangle
.
\]
Using the nonnegativity by \eqref{ass:lambdaAtK0Positive} we get the result.

Uniqueness follows by the monotonicity of $\Lambda_\rho'(u_\rho-\psi)\colon H^1_0(\Omega) \to H^{-1}(\Omega)$, which is a consequence of the fact that the derivative of a monotone map at a particular point is also monotone {(see e.g. \cite[Footnote 5]{AHRW})}.
\end{proof}

In what follows,
we show that $\Lambda_\rho'(u_\rho-\psi)$
can be represented by a function in $L^\infty(\Omega)$
under the assumption that the directional derivative
$S_\rho'(f_\rho)$
is linear. {The said linearity holds in the smooth case since then $\Lambda_\rho'(u_\rho-\psi)$ is linear and $S_\rho'(f_\rho)$ is the solution mapping of a linear PDE (see \cref{lem:boundOnDerivative}). In the nonsmooth case, linearity is attained at G\^ateaux points of $S_\rho$; further details will be given in \cref{sec:existence_limiting}.}

\begin{lem}\label{lem:linearity_etc_of_murho}
{Let \cref{ass:for_Bsw,ass:1} hold.}
If $f_\rho \in H^{-1}(\Omega)$ is given such that
$S_\rho'(f_\rho)$ is linear, i.e., $S_\rho'(f_\rho) \in \mathcal{L}(H^{-1}(\Omega),H^1_0(\Omega))$,
then
for a.a.\ $x \in \Omega$,
$\lambda_\rho'(x, \cdot)$ is differentiable at the point $S_\rho(f_\rho)(x) -\psi(x)$
and the derivative belongs to $[0, 1/\rho]$.
Consequently, the operator
$\Lambda_\rho'(S_\rho(f_\rho) -\psi) \colon L^2(\Omega) \to L^2(\Omega)$
is linear and can be identified with a function in $L^\infty(\Omega)$.
\end{lem}
\begin{proof}
	The operator
	\begin{equation*}
		S_\rho'(f_\rho) \colon H^{-1}(\Omega) \to H_0^1(\Omega)
	\end{equation*}
	is the inverse of
	\begin{equation*}
		-\Delta + \Lambda_\rho'(S_\rho(f_\rho) - \psi) \colon H_0^1(\Omega) \to H^{-1}(\Omega)
		.
	\end{equation*}
	Since the former operator is linear by assumption,
	the latter operator is linear as well.
	Consequently,
	$\Lambda_\rho'(S_\rho(f_\rho) -\psi)$
	is a linear operator.
	From \eqref{eq:Lambda_and_lambda}
	we infer the differentiability of $\lambda_\rho'(x, \cdot)$
	at $S_\rho(f_\rho)(x) -\psi(x)$
	and
	then \eqref{eq:formula_for_dir_der}, \eqref{eq:lambdaRhoDerivIncIfC1}
	imply $\lambda_\rho'(\cdot, S_\rho(f_\rho)-\psi) \in L^\infty(\Omega)$.
\end{proof}

In this linear case, we can define a measure $\mu_\rho$
via
\begin{equation}
	\label{eq:defn_mu_rho_intro}
	\mu_\rho(B)
	:=
	\int_B \Lambda_\rho'(S_\rho(f_\rho) - \psi) \d x
\end{equation}
and \eqref{eq:pdeForDerivativeStart}
becomes
\begin{equation}
	\label{eq:pdeForDerivative}
	-\Delta \alpha_\rho + \alpha_\rho \mu_\rho = d
	,
\end{equation}
with an appropriate interpretation of the term $\alpha_\rho \mu_\rho$.
We will use the notion of capacitary measures
to pass to the limit in this equation,
see \cref{sec:char_limits}.

\subsection{Convergence results}
We now consider the limit $\rho \to 0$.
The convergence result
\[u_\rho \to u := S(f)\qquad\text{in } H_0^1(\Omega)\]
has been shown for various choices of the penalty function
in, e.g., \cite[Lemma 3.3 and Lemma 3.5]{AHRW},
\cite[Theorem 3.1]{IK}, \cite[Theorem 4.1]{IKOCVIs}, \cite[Theorem 2.1]{IKH1} and \cite[Theorem~2.10]{SchielaWachsmuth}.
We demonstrate that this convergence
already holds under our general setting. In particular, our results are more general than currently available in the literature and are phrased in the natural function spaces for the problem (e.g., the source terms are only assumed to converge in the dual space).
We begin with a preparatory lemma.
\begin{lem}\label{lem:SW_estimate}
{Let \cref{ass:for_Bsw,ass:1} hold.} 
    For all $f_\rho \in H^{-1}(\Omega)$
	and $\rho \in (0,1]$,
	we have
\begin{align}
		\norm{(S_\rho(f_\rho)-\psi)^+}{L^2(\Omega)}
		&\le
		\frac{\sqrt{\rho}}{2}
		\left(\norm{f_\rho + \Delta\psi}{H^{-1}(\Omega)}+C_P C\right),
		\label{eq:limit_gives_feasibility_new}\\
		\norm{\Lambda_\rho(S_\rho(f_\rho) -\psi)}{L^2(\Omega)}
		&\leq
		\frac{1}{2\sqrt{\rho}}
		\left(\norm{f_\rho + \Delta\psi}{H^{-1}(\Omega)}+C_P C\right) + C,
		\label{eq:uniform_bound_Lambda_rho_new}
\end{align}
where  $C$ is the constant from \eqref{ass:kgrowth}.
\end{lem}
\begin{proof}
Denoting $u_\rho := S_\rho(f_\rho)$, we test the equation
\[-\Delta(u_\rho-\psi) + \Lambda_\rho(u_\rho-\psi) = f_\rho+\Delta\psi\]
with $(u_\rho-\psi)^+ \in H^1_0(\Omega)$. Using the fact that for $z \in H^1_0(\Omega)$, we have $\langle -\Delta z, z^+ \rangle = \int_\Omega \grad z \cdot \grad z^+ = \int_\Omega (\grad z^+ - \grad z^-)\cdot \grad z^+ = \int_\Omega |\grad z^+|^2$ (see e.g. \cite[Theorem A.1, Appendix A, \S II]{KinderlehrerStampacchia1980}), one has 
\begin{align*}
\langle -\Delta(u_\rho-\psi), (u_\rho-\psi)^+ \rangle = \norm{(u_\rho-\psi)^+}{H^1_0(\Omega)}^2,
\end{align*}
and using  \eqref{eq:new_consequence} we obtain
\begin{align*}
\int_\Omega \Lambda_\rho(u_\rho-\psi) (u_\rho-\psi)^+ \d x
\geq
\frac{1}{\rho}\norm{(u_\rho-\psi)^+}{L^2(\Omega)}^2 - \frac{1}{\rho}\int_\Omega K_\rho(u_\rho-\psi)^+ \d x
,
\end{align*}
which  leads to
	\begin{align}
		\MoveEqLeft
		\norm{(u_\rho-\psi)^+}{H^1_0(\Omega)}^2 + \frac{1}{\rho}\norm{(u_\rho-\psi)^+}{L^2(\Omega)}^2
		\leq \dual[\Big]{ f_\rho+\Delta\psi+\frac{K_\rho}{\rho} }{ (u_\rho-\psi)^+ }\label{eq:first_inequality}\\
		\nonumber &\leq \frac14 \norm{f_\rho + \Delta\psi+\frac{K_\rho}{\rho}}{H^{-1}(\Omega)}^2 + \norm{(u_\rho-\psi)^+}{H_0^1(\Omega)}^2
		,
	\end{align}
	where we used Young's inequality.
	This yields
	\begin{equation*}
		\norm{(u_\rho-\psi)^+}{L^2(\Omega)}
		\le
		\frac{\sqrt{\rho}}{2}
		\norm{f_\rho + \Delta\psi+\frac{K_\rho}{\rho}}{H^{-1}(\Omega)}
	\end{equation*}
	and \eqref{eq:limit_gives_feasibility_new} follows from Poincaré's inequality.
	Inequality \eqref{eq:uniform_bound_Lambda_rho_new} follows with \eqref{eq:NiceEstimateForLambda}.
\end{proof}

\begin{remark}
By arguing as in the proof of \cite[Lemma 2.3]{SchielaWachsmuth} we can obtain a better rate than above if we have additional regularity. Indeed, we can show that
if $f_\rho+\Delta\psi+ K_\rho\slash \rho \in L^2(\Omega)$ and $\rho \in (0,1]$, we have that
\begin{align}
\norm{(S_\rho(f_\rho)-\psi)^+}{L^2(\Omega)}
&\leq \rho\norm{f_\rho+\Delta\psi+\frac{K_\rho}{\rho}}{L^2(\Omega)}\label{eq:limit_gives_feasibility}\\
\norm{\Lambda_\rho(S_\rho(f_\rho) -\psi)}{L^2(\Omega)} &\leq   \norm{f_\rho+\Delta\psi+\frac{K_\rho}{\rho}}{L^2(\Omega)} + C\label{eq:uniform_bound_Lambda_rho}
\end{align}
where $C$ is the constant from \eqref{ass:kgrowth}. To see this, we can adapt \eqref{eq:first_inequality} as follows, making use of the assumed $L^2$ regularity:
\begin{align*}
\MoveEqLeft
\norm{(u_\rho-\psi)^+}{H^1_0(\Omega)}^2 + \frac{1}{\rho}\norm{(u_\rho-\psi)^+}{L^2(\Omega)}^2
\leq \dual[\Big]{ f_\rho+\Delta\psi+\frac{K_\rho}{\rho}}{ (u_\rho-\psi)^+ }\\
&\leq \norm{f_\rho+\Delta\psi+\frac{K_\rho}{\rho}}{L^2(\Omega)}\norm{(u_\rho-\psi)^+}{L^2(\Omega)},
\end{align*}
whence \eqref{eq:limit_gives_feasibility}.  The second estimate \eqref{eq:uniform_bound_Lambda_rho} follows as before.
\end{remark}

\begin{prop}\label{prop:convergence_to_VI_soln}
{Let \cref{ass:for_Bsw,ass:1} hold.}
Let $f_\rho \to f$ in $H^{-1}(\Omega)$.
We have $u_\rho = S_\rho(f_\rho) \to S(f)=:u$ in $H^1_0(\Omega)$, i.e., the limit $u$ solves the VI \eqref{eq:VI}.
\end{prop}
\begin{proof}
The estimate in \cref{lem:existence_lipschitz_urho} implies that $\{u_\rho\}$ is uniformly bounded in $H^1_0(\Omega)$,
hence there exists a $v \in H^1_0(\Omega)$ such that for a subsequence (that we relabel), we have
\[u_\rho \weaklyto  v \quad\text{in $H^1_0(\Omega)$}.\]
In what follows, we argue that the convergence is strong and that the limit $v$ equals $u$.
Hence, the entire sequence converges towards $u := S(f)$ in $H_0^1(\Omega)$.

We first note that \eqref{eq:limit_gives_feasibility_new}
implies the feasibility of $v$, i.e., $v \le \psi$.
Consequently,
we can use $v$ in the VI \eqref{eq:VI}
and test
the PDE for $u_\rho$ with $u_\rho-u$.
Adding the resulting expressions
gives
\begin{align*}
\langle -\Delta u_\rho, u_\rho - u \rangle \leq \langle f_\rho, u_\rho - u \rangle + \langle -\Delta u -f, v - u \rangle - \langle \Lambda_\rho(u_\rho - \psi), u_\rho - u \rangle. 
\end{align*}
Subtracting $\langle -\Delta u, u_\rho -u \rangle$ from both sides and adding and subtracting $\langle f, u_\rho - u \rangle$ on the right-hand side,  we end up with
\begin{equation*}
	\norm{u - u_\rho}{H_0^1(\Omega)}^2
	\leq
	\dual{f - f_\rho}{u - u_\rho}
	+
	\dual{-\Delta u - f}{v - u_\rho}
	-
	\dual{\Lambda_\rho(u_\rho - \psi)}{u_\rho - u}
	.
\end{equation*}
It remains to check that the last addend on the right-hand side converges to $0$.
From $v-\psi + \kzero_\rho \leq \kzero_\rho$
we get $\Lambda_\rho(v-\psi + \kzero_\rho) = 0$,
see \eqref{eq:structure_lambda}.
Together with the monotonicity of $\Lambda_\rho$,
we obtain
\begin{align*}
	\dual{\Lambda_\rho(u_\rho-\psi) }{ u_\rho-v }
	&=
	\dual{ \Lambda_\rho(u_\rho-\psi) - \Lambda_\rho(v-\psi+\kzero_\rho) }{ u_\rho - (v + \kzero_\rho) + \kzero_\rho }
	\\
	&\geq
	\dual{\Lambda_\rho(u_\rho-\psi) }{ \kzero_\rho }
	.
\end{align*}
Combining \eqref{eq:uniform_bound_Lambda_rho_new} with \eqref{ass:kgrowth}
yields that the right-hand side converges to $0$.
This can be used in the above estimate
and the claim follows.
\end{proof}

Now, we can revisit the estimate in \cref{lem:SW_estimate} and improve the rate implied by \eqref{eq:limit_gives_feasibility_new}. This will be crucial later on.
\begin{lem}\label{lem:improved_rate}
{Let \cref{ass:for_Bsw,ass:1} hold.}
Let $f_\rho \to f$ in $H^{-1}(\Omega)$. Then,
\[\rho^{-1/2} \norm{ (S_\rho(f_\rho) - \psi)^+ }{L^2(\Omega)} \to 0.\] 
\end{lem}
\begin{proof}
Thanks to \cref{prop:convergence_to_VI_soln}, we have $(u_\rho - \psi)^+ \to 0$   in $H_0^1(\Omega)$
with $u_\rho := S_\rho(f_\rho)$.
In the inequality \eqref{eq:first_inequality}, i.e.,
\[	\norm{(u_\rho-\psi)^+}{H^1_0(\Omega)}^2 + \frac{1}{\rho}\norm{(u_\rho-\psi)^+}{L^2(\Omega)}^2
\leq \dual[\Big]{ f_\rho +\Delta\psi+ \frac{K_\rho}{\rho}}{ (u_\rho-\psi)^+ },\]
the duality product on the right-hand side clearly goes to zero, see \eqref{ass:kgrowth}.
This shows the assertion.
\end{proof}

\section{Pointwise characterisation of the limit of the derivatives}\label{sec:fixed_d}

In this section,
we study the convergence of the (directional) derivatives
from \cref{lem:boundOnDerivative}
and provide properties for the limits.

\begin{prop}\label{prop:convergence_derivatives}
{Let \cref{ass:for_Bsw,ass:1} hold.}
There exists an $\alpha \in H^1_0(\Omega)$ such that, for a subsequence (that we relabel) the derivative
satisfies
\[ S_\rho'(f_\rho)(d) \weaklyto \alpha \quad\text{in $H^1_0(\Omega)$}.\]
\end{prop}
\begin{proof}
	This follows directly from the bound in
	\cref{lem:boundOnDerivative}.
\end{proof}

We want to know what properties the limit $\alpha$  satisfies.

\subsection{A complementarity condition on \texorpdfstring{$\alpha$}{alpha}}
We start with a simple lemma.
\begin{lem}\label{lem:calcLemma}
If $\sigma\colon \mathbb{R} \to \mathbb{R}$ is $C^1$ and convex,
and satisfies $\sigma =0$ on $(-\infty, r_0]$, we have 
\[\sigma(r)  \leq (r-r_0)\sigma'(r) \quad\forall r \in \R.\]
\end{lem}
\begin{proof}
As $\sigma$ is a differentiable convex function, it satisfies $\sigma(r) \leq \sigma(s) + (r-s)\sigma'(r)$ for all $r,s \in \R$. The result follows from the fact that $\sigma(r_0)=0$.
\end{proof}

In the smooth setting where $\lambda_\rho(x,\cdot)$ is $C^1$, we can take in \cref{lem:calcLemma} $\sigma(\cdot) = \lambda_\rho(x, \cdot)$ for fixed $x$ and obtain 
\begin{equation}
\lambda_\rho(x, r) \leq (r-\kzero_\rho(x, \rho))\lambda_\rho'(x, r)
\quad \forall r \in [\kzero_\rho(x), \kone_\rho(x)].
\label{eq:calcLemmaEstimate}
\end{equation}
Note that this inequality also holds in the nonsmooth case,
since this implies $\kzero_\rho \equiv \kone_\rho$.
\begin{prop}\label{lem:orthogonality}
{Let \cref{ass:for_Bsw,ass:1} hold.}
Let $f_\rho \to f$ in $H^{-1}(\Omega)$.
With $\xi := f+\Delta u$ and $\alpha$ denoting the weak limit of (a subsequence of) $S_\rho'(f_\rho)(d)$ from \cref{prop:convergence_derivatives}, we have
\[\langle \xi, \alpha \rangle = \langle \xi, \alpha^+ \rangle = \langle \xi, \alpha^- \rangle = 0.\]
\end{prop}
For a characterisation of this result in terms of the strictly active set, see \eqref{eq:harderwachsmuth_characterisation}.
\begin{proof}
Set $u_\rho := S_\rho(f_\rho)$ and define the sets\footnote{{These sets, being inverse images of measurable sets under measurable functions, are also measurable.}}
\begin{equation*}
M_0 = \{ u_\rho-\psi \leq \kzero_\rho \}, \quad
M_1 = \{\kzero_\rho < u_\rho -\psi { < } \kone_\rho\}, \quad
M_2 = \{ \kone_\rho \leq u_\rho - \psi\}.
\end{equation*}
Note that $M_1 = \emptyset$ in the nonsmooth case.
Further,
\eqref{ass:k1PlusJNonNeg} implies
\begin{equation}
\text{$u_\rho-\psi + \jj_\rho \geq 0$ on $M_2$}.\label{eq:consequenceOfSomething}
\end{equation}
Define $\xi_\rho := f_\rho+\Delta u_\rho$.  We have
\begin{align*}
\langle \xi_\rho, \alpha_\rho \rangle  
&=  \int_\Omega \Lambda_\rho(u_\rho-\psi)\alpha_\rho \d x\\
&=\int_{M_1} \Lambda_\rho(u_\rho-\psi)\alpha_\rho \d x+  \int_{M_2} \frac{u_\rho-\psi + \jj_\rho}{\rho} \alpha_\rho \d x\\
&=\int_{M_1} \Lambda_\rho(u_\rho-\psi)\alpha_\rho \d x+  \int_{M_2} \frac{(u_\rho-\psi + \jj_\rho)^+}{\rho} \alpha_\rho \d x
\end{align*}
with no integral over $M_0$ because $\lambda_\rho(x,\cdot)$ vanishes  on $(-\infty, \kzero_\rho(x)]$
and the positive part is introduced due to \eqref{eq:consequenceOfSomething}.

Regarding the first term, using the estimate \eqref{eq:calcLemmaEstimate}, we have 
\begin{align*}
\left|\int_{M_1} \Lambda_\rho(u_\rho-\psi)\alpha_\rho\right| \d x
&\leq  \int_{M_1} \lambda_\rho(\cdot, u_\rho-\psi)|\alpha_\rho| \d x\\
&\leq \int_{M_1} (u_\rho-\psi- \kzero_\rho)\lambda'_\rho(\cdot, u_\rho-\psi)|\alpha_\rho| \d x\\
&\leq  \frac{\norm{\kone_\rho - \kzero_\rho}{L^2(\Omega)}}{\sqrt{\rho}}\norm{\sqrt{\lambda_\rho'(\cdot, u_\rho-\psi)}\alpha_\rho\chi_{M_1}}{L^2(\Omega)}
\end{align*}
where we used the bounds on $\lambda_\rho'$ given in \eqref{eq:lambdaRhoDerivIncIfC1} to justify the last estimate.
As for the second integral, we obtain
\begin{align*}
\left|\int_{M_2} \frac{(u_\rho-\psi + \jj_\rho)^+}{\rho} \alpha_\rho \d x\right|
&\leq \frac{\norm{(u_\rho-\psi + \jj_\rho)^+\chi_{M_2}}{L^2(\Omega)}}{\sqrt{\rho}} \frac{\norm{\alpha_\rho \chi_{M_2}}{L^2(\Omega)}}{\sqrt{\rho}}.
\end{align*}
Thus
\begin{align}
\nonumber|\langle \xi_\rho, \alpha_\rho \rangle|   &\leq \frac{\norm{\kone_\rho - \kzero_\rho}{L^2(\Omega)}}{\sqrt{\rho}}\norm{\sqrt{\lambda_\rho'(\cdot, u_\rho-\psi)}\alpha_\rho\chi_{M_1}}{L^2(\Omega)}\\
&\quad + \frac{\norm{(u_\rho-\psi + \jj_\rho)^+\chi_{M_2}}{L^2(\Omega)}}{\sqrt{\rho}} \frac{\norm{\alpha_\rho \chi_{M_2}}{L^2(\Omega)}}{\sqrt{\rho}}.\label{eq:productInequality}
\end{align}
By \eqref{ass:kgrowth}, the first multiplicand in the first term above vanishes as $\rho \to 0$. Let us show that the first multiplicand in the second term also vanishes.
By using the triangle inequality and the fact that $(a+b)^+ \leq a^+ + b^+$ for all $a,b \in \mathbb{R}$,
\begin{align*}
	\frac{\norm{(u_\rho-\psi + \jj_\rho)^+\chi_{M_2}}{L^2(\Omega)}}{\sqrt{\rho}}
	& \le
	\frac{\norm{ (u_\rho - \psi)^+ }{L^2(\Omega)}}{\sqrt{\rho}}
	+
	\frac{\norm{ \jj_\rho }{L^2(\Omega)}}{\sqrt{\rho}}
\end{align*}
and both terms on the right-hand side vanish by \cref{lem:improved_rate} and \eqref{eq:jjgrowth} respectively.

Now we simply need to show that second multiplicand in the two terms on the right-hand side of \eqref{eq:productInequality} are bounded.  Since $\alpha_\rho$ is bounded {(in $H^1_0(\Omega)$)}, so is $d+\Delta\alpha_\rho$ {(in $H^{-1}(\Omega)$)}.
Hence, {for a constant $C$ independent of $\rho$, we have}
\begin{align*}
C &\geq |\langle d+ \Delta\alpha_\rho, \alpha_\rho \rangle|
= \int_\Omega \Lambda_\rho'(u_\rho-\psi)(\alpha_\rho)\alpha_\rho \d x\\
&\ge\int_{M_1}\lambda_\rho'(\cdot, u_\rho-\psi)(\alpha_\rho)\alpha_\rho \d x
+ \int_{M_2}\lambda_\rho'(\cdot, u_\rho-\psi)(\alpha_\rho)\alpha_\rho \d x,
\end{align*}
where we used that the integral over $M_0$ is nonnegative
due to \eqref{ass:lambdaAtK0Positive}.
Now,
on $M_1$, the derivative is linear w.r.t.\ the direction
and on $M_2$, the derivative is $1/\rho$.
Thus, we continue with
\begin{align*}
	C
&\geq  \int_{M_1}\lambda_\rho'(\cdot, u_\rho-\psi)\alpha_\rho^2 \d x + \frac 1\rho \int_{M_2}\alpha_\rho^2 \d x \\
&= \norm{\sqrt{\lambda_\rho'(\cdot, u_\rho-\psi)}\alpha_\rho\chi_{M_1}}{L^2(\Omega)}^2 + \frac 1\rho \norm{\alpha_\rho\chi_{M_2}}{L^2(\Omega)}^2.
\end{align*}
Plugging this information back into  \eqref{eq:productInequality}, we find that $\langle \xi_\rho, \alpha_\rho \rangle \to 0$ and hence, since $\alpha_\rho \weaklyto \alpha$ in $H^1_0(\Omega)$ and  $\xi_\rho = f_\rho + \Delta u_\rho \to f + \Delta u = \xi$ in $H^{-1}(\Omega)$,
\[\langle \xi, \alpha \rangle = 0.\]
In order to produce the same identity for $\alpha^+$,
one can repeat exactly the same calculation with $\alpha_\rho^+$ instead of $\alpha_\rho$.
In the final step,
we can use that
$\alpha_\rho \weaklyto \alpha$ in $H_0^1(\Omega)$
implies
$\alpha_\rho^+ \weaklyto \alpha^+$ in $H_0^1(\Omega)$.
This yields the claim.
\end{proof}

\subsection{An orthogonality condition on \texorpdfstring{$-\Delta\alpha-d$}{-Delta alpha - d}}\label{sec:orthogonality_lap_alpha}
In this section we will show that $-\Delta\alpha -d$ satisfies an orthogonality condition involving the coincidence set. {First, we introduce some notions from potential theory. The \textit{capacity} of a set $O \subset \Omega$ is defined as
\begin{equation*}
\mathrm{cap}(O) := \inf \left\{\norm{\grad v}{L^2(\Omega)}^2 : v \in H^1_0(\Omega)\text{ and } v \geq 1 \text{ a.e. in a neighbourhood of $O$}\right\}.
\end{equation*}
A set could have measure zero but positive capacity and in general the measure of a set is bounded from above by a constant times its capacity; thus the capacity is a finer way to `measure' sets. 

A function $v \colon \Omega \to \mathbb{R}$ is called \textit{quasi-continuous} if for all $\epsilon > 0$, there exists an open set $G_\epsilon \subset \Omega$ with $\mathrm{cap}(G_\epsilon) < \epsilon$ and $v$ is continuous on $\Omega \setminus G_\epsilon$. A set $O \subset \Omega$ is called \textit{quasi-open} if for all $\epsilon > 0$, there exists an open set $G_\epsilon \subset \Omega$ with $\mathrm{cap}(G_\epsilon) < \epsilon$ and $O \cup G_\epsilon$ is open. A set is \textit{quasi-closed} if its complement in $\Omega$ is quasi-open. 

We say that a property $P(x)$ depending on $x \in \Omega$ \textit{holds quasi-everywhere} if it holds everywhere except on a set of capacity zero, i.e., if 
\[\mathrm{cap}(\{x \in \Omega : P(x) \text{ does not hold}\}) = 0,\]
and we write q.e. in place of quasi-everywhere.  Note that:
\begin{itemize}\itemsep=0em
\item If $v$ is quasi-continuous, the set $\{ v > 0\}$ is quasi-open. 
\item Every $v \in H^1(\Omega)$ possesses a quasi-continuous representative which we can identify $v$ with (we will do this from now on). 
	This representative is unique up to subsets of capacity zero.
\item Consequently, sets such as $\{ v \geq 0\}$ with $v \in H^1(\Omega)$ are defined up to sets of capacity zero by using the quasi-continuous representative.
\item Every sequence which which converges in $H^1_0(\Omega)$ possesses a subsequence that converges pointwise q.e.
\end{itemize}
For more details, see \cite[\S 6.4.3]{MR1756264}, and for an introduction tailored to optimal control, we refer to \cite[Section~3]{HarderWachsmuth}.}

\begin{lem}
	\label{lem:almost_complementarity_condition_purple}
{Let \cref{ass:for_Bsw,ass:1} hold.}
We have
	\begin{equation*}
		\dual{-\Delta\alpha_\rho-d}{v} = 0
		\quad
		\forall v \in H^1_0(\Omega), \;\; v = 0 \text{ q.e. on $\{(u_\rho- \psi + \norm{\kzero_\rho}{L^\infty(\Omega)})^- =0\}$}
		.
	\end{equation*}
\end{lem}
Note that the set
$\{(u_\rho- \psi + \norm{\kzero_\rho}{L^\infty(\Omega)})^- =0\}$
is defined up to sets of capacity zero,
since $u_\rho - \psi + \norm{\kzero_\rho}{L^\infty(\Omega)} \in H^1(\Omega)$
is quasicontinuous, see \cite[Theorem 6.1, \S 8]{MR2731611}.
 \begin{proof}
Taking $v \in H^1_0(\Omega)$ such that   $v = 0$ q.e. on $\{(u_\rho- \psi + \norm{\kzero_\rho}{L^\infty(\Omega)})^- =0\}$,  we have
\begin{align}
\langle -\Delta\alpha_\rho -d, v \rangle
= -\int_{\{u_\rho-\psi \geq \kzero_\rho\}} \Lambda_{\rho}'(u_\rho-\psi)(\alpha_\rho) v \d x
\label{eq:pre_1_purple_new} 
\end{align}
because $\Lambda_\rho'(u_\rho-\psi)$ vanishes when $u_\rho-\psi \leq \kzero_\rho$.
By definition, we have that $v=0$ q.e.\ on the set $\{u_\rho- \psi + \norm{\kzero_\rho}{L^\infty(\Omega)} \geq 0\}$.
Note that  $u_\rho-\psi+\norm{\kzero_\rho}{L^\infty(\Omega)} \geq u_\rho-\psi - \kzero_\rho$ a.e.\ and so we have the inclusion
\[\{u_\rho-\psi - \kzero_\rho \geq 0\} \subset \{u_\rho-\psi+\norm{\kzero_\rho}{L^\infty(\Omega)}  \geq 0\}.\]
It follows that $v=0$ a.e. on $\{u_\rho-\psi - \kzero_\rho \geq 0\}$ too.
Using this in \eqref{eq:pre_1_purple_new} we obtain the desired statement.
\end{proof}

\begin{prop}\label{prop:multiplerConditionsNew_purple}
{Let \cref{ass:for_Bsw,ass:1} hold.}
We have
\begin{align*}
 \langle -\Delta\alpha-d, v \rangle &= 0  \qquad \text{$\forall v \in H^1_0(\Omega) : v = 0$ q.e. on $\{u=\psi\}$}.
\end{align*}
\end{prop}
\begin{proof}
By \eqref{ass:k0ConvergesToZeroLinf_purple} it follows that $u_\rho- \psi + \norm{\kzero_\rho}{L^\infty(\Omega)} \to u-\psi$ in $H^1(\Omega)$.
We want to show that $s_\rho := (u_\rho- \psi + \norm{\kzero_\rho}{L^\infty(\Omega)})^-$
converges to $s := (u - \psi)^- = \psi-u$ in capacity{, i.e.,
\[\capa(\{|s_\rho-s| \geq \varepsilon\}) \to 0\quad\text{as } \rho \to 0\]
for all $\varepsilon > 0$.}

{As $u_\rho \to u$ in $H^1_0(\Omega)$ (due to \cref{prop:convergence_to_VI_soln}), we can apply} \cite[Lemma 2.2]{TowardsM} to get that $u_\rho$ converges towards $u$ in capacity. 
Let $\varepsilon > 0$ be arbitrary.
By \eqref{ass:k0ConvergesToZeroLinf_purple},
there exists $\rho_0 > 0$ such that $\rho \leq \rho_0$ implies $\norm{\kzero_\rho}{L^\infty(\Omega)} \leq \frac{\varepsilon}{2}$.
Restricting to such sufficiently small $\rho$, we have
\begin{equation*}
	\abs{ s_\rho - s }
	\le
	\abs{ u_\rho + \norm{\kzero_\rho}{L^\infty(\Omega)} - u }
	\le
	\abs{ u_\rho - u } + \frac\varepsilon2.
\end{equation*}
Consequently,
\[
\capa(\{|s_\rho-s| \geq \varepsilon\})
\le
\capa(\{|u_\rho-u| \geq \varepsilon/2\})
\to 0
\quad\text{as } \rho \to 0.
\]
So we have shown that $s_\rho \to s$ in capacity as desired.

Now, the result of \cref{lem:almost_complementarity_condition_purple} implies  
\[\pm(-\Delta\alpha_\rho -d) \in \{ v \in H^1_0(\Omega) : v \geq 0 \text{ q.e. on $\Omega$ and } v=0 \text{ q.e. on } \{s_\rho = 0 \}\}^\circ\]
with $s_\rho = (u_\rho-\psi+ \norm{\kzero_\rho}{L^\infty(\Omega)})^-$ as above {and we used the notation $^\circ$ to denote the polar cone, i.e., if $M \subset H^1_0(\Omega)$ is a set then $M^\circ := \{ f \in H^{-1}(\Omega) : \langle f, v \rangle \leq 0 \text{ for all $v \in M$}\}$.}

{Bearing in mind that $(-\Delta\alpha_\rho-d) \weaklyto (-\Delta \alpha-d)$ in $H^{-1}(\Omega)$, utilising} \cite[Lemma 2.6]{TowardsM} we find
\[\pm(-\Delta\alpha -d) \in \{ v \in H^1_0(\Omega) : v \geq 0 \text{ q.e. on $\Omega$ and } v=0 \text{ q.e. on } \{s= 0 \} \}^\circ\]
where $s=(u-\psi)^- = \psi-u$ is the limit (w.r.t.\ convergence in capacity) of $s_\rho$.
This yields the desired statement after using the decomposition $v=v^+-v^-$.
\end{proof}

\begin{remark}
	The proof of \cref{lem:almost_complementarity_condition_purple} even shows
	\begin{equation*}
		\dual{-\Delta\alpha_\rho-d}{v} = 0
		\quad
		\forall v \in H^1_0(\Omega), \;\; v = 0 \text{ a.e. on $\{(u_\rho- \psi + \norm{\kzero_\rho}{L^\infty(\Omega)})^- =0\}$}
		.
	\end{equation*}
	However, one cannot generalise (the proof of) \cref{prop:multiplerConditionsNew_purple}
	to obtain
	\begin{align*}
		\langle -\Delta\alpha-d, v \rangle &= 0  \qquad \text{$\forall v \in H^1_0(\Omega) : v = 0$ a.e. on $\{u=\psi\}$}.
	\end{align*}
	The problem is the following.
	The set $\{u = \psi\}$ could have measure zero but positive capacity.
	In this case, the assumption $v = 0$ a.e.\ on $\{u = \psi\}$ is void.
	Consequently, $v$ could be constant $1$ in a neighbourhood of $\{u = \psi\}$.
	If further $v_\rho \to v$ in $H_0^1(\Omega)$ is given
	(this is needed for passing to the limit with $\dual{-\Delta\alpha_\rho-d}{v}$),
	we have $v_\rho \to v$ in capacity.
	Further, $\{u_\rho - \psi + \norm{\kzero_\rho}{L^\infty} \ge 0\}$
	could contain an open neighbourhood of $\{u = \psi\}$.
	Therefore, $v_\rho = 0$ a.e.\ on $\{u_\rho - \psi + \norm{\kzero_\rho}{L^\infty} \ge 0\}$
	implies
	$v_\rho = 0$ q.e.\ on $\{u = \psi\}$.
	This contradicts 
	$v_\rho \to v$ in capacity.

\end{remark}

If $\psi \in H^1_0(\Omega)$, the argument above can be simplified and we could also alternatively have argued in a similar way to \cite[\S 7]{AHRW}.

\section{Generalised derivatives as limits}
In the convergence results of the previous section, the direction $d$ was fixed.  We need something stronger than this: we would like limiting statements for the derivative when seen as a linear operator, and furthermore, we would like the subsequence that converges to be independent of $d$. 

\subsection{Existence of limiting elements}\label{sec:existence_limiting}
First, we prove that the Gâteaux derivatives
converge in the weak operator topology.
\begin{prop}\label{prop:existenceQGeneral}
        {Let \cref{ass:for_Bsw,ass:1} hold.}
	Let a sequence $\{f_\rho\} \subset H^{-1}(\Omega)$
	be given such that
	$f_\rho \to f$ in $H^{-1}(\Omega)$
	and such that $S_\rho$ is Gâteaux differentiable at $f_\rho$ for all $\rho > 0$.
	Then,
	there exists a map $\Qf \in \mathcal{L}(H^{-1}(\Omega),H^1_0(\Omega))$ such that for a subsequence (that we relabel),
\begin{align*}
S_\rho'(f_\rho) &\WOT \Qf.
\end{align*}
\end{prop}
\begin{proof}
The bound of \cref{lem:boundOnDerivative} shows that 
\[\norm{S_\rho'(f_\rho)}{\mathcal{L}(H^{-1}(\Omega),H^1_0(\Omega))} \leq 1\]
for all $\rho > 0$.
By the separability of $H^1_0(\Omega)$, one can show that the unit ball is sequentially compact in the WOT,
see also \cite[Remark~3.1]{CCMW}. The result follows.
\end{proof}
Note that the Gâteaux differentiability  assumption of the previous result holds automatically in the smooth case;
this follows from \eqref{eq:Lambda_and_lambda}.
The question remains whether it can also hold in the nonsmooth case.
We can in fact use the generalised Rademacher's theorem of Mignot \cite[Theorem 1.2]{Mignot},
which tells us that for each $\rho$, there exists a dense set on which $S_\rho$ is Gâteaux differentiable.
However, this dense set depends on $\rho$,
and  because we wish to consider limits it would be ideal to have \emph{one} dense set on which $S_\rho$ is Gâteaux
for all $\rho$ (or at least for a sequence $\{\rho_n\}$).
We cannot simply take the intersection because an intersection of dense sets may not be dense.
Thus we need another argument where we essentially embed $\{S_\rho\}$ into an infinite dimensional space,
which we give next.
\begin{lem}\label{lem:density_result_general}
        {Let \cref{ass:for_Bsw,ass:1} hold.}
	Let a sequence $\{\rho_n\} \subset (0,1]$ with $\rho_n \to 0$ be given
	and, for brevity, set $S_n := S_{\rho_n}$.
	Then,
there exists a dense set $F \subset H^{-1}(\Omega)$ such that $S_n$ is Gâteaux differentiable on $F$ for all $n \in \mathbb{N}$.
\end{lem}
We remark that this result
also follows from the theory of Aronszajn null sets.
Indeed, because $S_\rho$ is Lipschitz, $S_\rho$ is G\^ateaux differentiable outside an Aronszajn null set \cite[Theorem 6.42]{MR1727673}
and the countable union of Aronszajn null sets remains Aronszajn null (see the paragraph after Definition 6.23 of \cite{MR1727673}).
This implies the result.
For convenience, however, we give a different and direct proof.
\begin{proof}
We denote by $\mathcal{V} := \ell^2(\mathbb{N};H^1_0(\Omega))$ the set of $\ell^2$-summable $H^1_0(\Omega)$-valued sequences.
Consider the map $H\colon H^{-1}(\Omega) \to \mathcal{V}$ defined by  
\[H(f) := \parens[\big]{ 2^{-1/2} S_1(f), 2^{-2/2} S_2(f), \hdots },\]
i.e., the $n$th component is $[H(f)]_n = 2^{-n/2} S_n(f)$.
Note that
\begin{equation*}
	\norm{H(f)}{\mathcal{V}}^2 = \sum_{n=1}^\infty \norm{2^{-n/2} S_n(f)}{H^1_0(\Omega)}^2  \leq C^2\sum_{n=1}^\infty 2^{-n} = C^2
	<\infty
	,
\end{equation*}
using the fact that $\norm{S_n(f)}{H^1_0(\Omega)} \leq C$ for a constant $C$
independent of $n$, see \cref{lem:existence_lipschitz_urho}.
Further, $H$ is Lipschitz, since
\begin{equation*}
\norm{H(f)-H(g)}{\mathcal{V}}^2
\leq  \sum_{n=1}^\infty 2^{-n} \norm{ (S_n(f)-S_n(g))}{H^1_0(\Omega)}^2
\le  \norm{f-g}{H^{-1}(\Omega)}^2,
\end{equation*}
where we used the fact that each $S_n$ is Lipschitz with constant $1$,
see \cref{lem:existence_lipschitz_urho}.
Since $\mathcal{V}$ is a Hilbert space and $H^{-1}(\Omega)$ is a separable Hilbert space, it follows \cite[Theorem 1.2]{Mignot} that $H$ is G\^ateaux differentiable on a dense set $F \subset H^{-1}(\Omega)$.
Now, it is straightforward to check that
the Gâteaux differentiability of $H$ on $F$
implies that $S_n$ is Gâteaux differentiable on $F$
for all $n \in \N$.
\end{proof}

Observe that this result holds for a countable sequence and does not and cannot hold for the original uncountable family $\{S_\rho\}$. This is  because the differentiability holds up to a set of measure zero and countable unions (of the exceptional sets) still have measure zero, which of course, does not apply for the uncountable case. A simple real-valued example is $S_\rho(x) := | x - \rho |.$ This function is not differentiable at $\rho$, so if one considers all of these functions, the common differentiability set is just the empty set.

\subsection{Capacitary measures and characterisation of limits}\label{sec:char_limits}

The measure $\mu_\rho$ defined in \eqref{eq:defn_mu_rho_intro}
has some additional regularity that we can exploit.
We begin with recalling the notion of capacitary measures, see, e.g., \cite[Definition 3.1]{RW}.
\begin{defn}[Capacitary measure]
A \emph{capacitary measure} is a Borel measure $\mu$ such that 
\begin{enumerate}[label=(\roman*)]\itemsep=0em
\item $\mu(B)=0$ for every Borel set $B \subset \Omega$ with $\capa(B) = 0$ 
\item $\mu$ is regular in the sense that every Borel set $B \subset \Omega$ satisfies
\[
	\mu(B) = \inf\set{ \mu(O) : O \text{ quasi-open and } \capa(B \setminus O) = 0}.
\]
\end{enumerate}
We denote by $\M_0(\Omega)$ the set of all capacitary measures on $\Omega$.
\end{defn}
Note that a measure which is absolutely continuous w.r.t.\ the Lebesgue measure
is a capacitary measure,
see also \cite[Remark 3.2 and Definition 2.1]{DalMaso}.
In particular,
\eqref{eq:defn_mu_rho_intro}
defines a capacitary measure in the setting of \cref{lem:linearity_etc_of_murho}.

From now on,
given (an equivalence class) $v \in H_0^1(\Omega)$,
we will always work with a Borel measurable
and quasicontinuous representative.
Note that such a representative is uniquely determined
up to subsets of capacity zero which are $\mu$-nullsets for every $\mu \in \M_0(\Omega)$.

We
denote by $L^2_\mu(\Omega)$
the usual Lebesgue space w.r.t.\ a measure $\mu \in \M_0(\Omega)$.

For $\mu \in \M_0(\Omega)$
and $f \in H^{-1}(\Omega)$,
one can check {simply by the Lax--Milgram lemma (see also \cite[\S 3]{RW})} that there is a unique solution $y \in H_0^1(\Omega) \cap L^2_\mu(\Omega)$
of
\begin{equation}
	\label{eq:defn_Lmu}
	\dual{ -\Delta y }{v} + \int_\Omega y v \d\mu
	=
	\dual{f}{v}
	\qquad\forall v \in H_0^1(\Omega) \cap L^2_\mu(\Omega)
	.
\end{equation}
The solution map of this weak formulation is denoted by
$L_\mu\colon H^{-1}(\Omega) \to H^1_0(\Omega)$,
i.e., $L_\mu f = y$.

Associated to the VI \eqref{eq:VI},
define the \emph{inactive set} (or non-coincidence set) $I:= \{u < \psi\}$
and the \emph{active set} (or coincidence set) $A:=\{u = \psi\}$.
We denote by $A_s \subset A$ the \emph{strictly active set}
which can be defined via the quasi-support (or fine support) of the measure
$\xi = f + \Delta u$; this is a quasi-closed subset.
Note that all these sets are defined up to subsets of capacity zero.

Under \cref{ass:for_Bsw} and if $S(f) \in C_0(\Omega)$,
the strong-weak generalised derivative
$\partial_B^{sw} S(f) \subset \mathcal{L}(H^{-1}(\Omega), H^1_0(\Omega))$
of $S$ at $f$
was characterised in \cite[Theorem 5.6]{RW}
and we have
\begin{equation}
	\label{eq:RW_characterisation}
	\partial_B^{sw} S(f) = \{ L_\mu : \mu \in \M_0(\Omega),\; \mu(I) = 0 ,\; \mu= \infty \text{ on } A_s\},
\end{equation}
where $\mu=\infty$ on $A_s$ is to be understood in the sense that
\[v=0 \quad\text{q.e. on $A_s$ for all $v \in H^1_0(\Omega) \cap L^2_\mu(\Omega)$}\]
(see \cite[Lemma 5.2]{RW} for some equivalent characterisations).
We will show that the limits of the derivatives from \cref{prop:existenceQGeneral} belong to this set.

In the linear setting of \cref{lem:linearity_etc_of_murho},
$\mu_\rho$ is a measure and the  equation \eqref{eq:pdeForDerivative} for $\alpha_\rho$
is to be understood in the weak form
\[
\langle -\Delta\alpha_\rho, z \rangle + \int_\Omega \alpha_\rho z \d\mu_\rho = \langle d, z \rangle \qquad \forall z \in H^1_0(\Omega)
.
\]
Note that \eqref{eq:defn_mu_rho_intro} implies
\[\int_\Omega \alpha_\rho z \d\mu_\rho =  \int_\Omega  \Lambda_\rho'(u_\rho-\psi)\alpha_\rho z\d x.\]
Since $\Lambda_\rho'(u_\rho - \psi)$
belongs to $L^\infty(\Omega)$,
see \cref{lem:linearity_etc_of_murho},
we have
$H_0^1(\Omega) \subset L^2_{\mu_\rho}(\Omega)$.
Consequently,
the above equation is equivalent to
$L_{\mu_\rho}(d)=\alpha_\rho$ and this implies
\begin{equation}
L_{\mu_\rho} \equiv S_\rho'(f_\rho).\label{eq:Lmurho_is_Srho_prime}
\end{equation}
To summarise,
if the directional derivative $S_\rho'(f_\rho)$
is linear,
we can find a capacitary measure $\mu_\rho$
such that
\eqref{eq:Lmurho_is_Srho_prime} holds.

We now define a notion of convergence for measures related to convergence in the weak operator topology (see \cref{defn:WOT}) of the associated solution operators $L_\mu$ defined in \eqref{eq:defn_Lmu}.
\begin{defn}
	Let $\{\mu_n\}$ and $\mu$ be capacitary measures.
We say that $\mu_n \gammato \mu$ if and only if
\(L_{\mu_n} \WOT L_\mu.
\)
\end{defn}

The sequential compactness of the space $\M_0(\Omega)$,
see, e.g., \cite[Theorem 4.14]{DalMasoMosco}, is a crucial property which implies the next result.
\begin{prop}
	\label{lem:existenceOfmuStar}
        {Let \cref{ass:for_Bsw,ass:1} hold.}
	Let a sequence $\{f_\rho\}$ be given
	such that $S_\rho'(f_\rho)$ is linear for all $\rho$.
	We denote by $\mu_\rho$ the associated capacitary measure,
	see \eqref{eq:defn_mu_rho_intro}.
	Then,
	there exists a capacitary measure $\mu \in \M_0(\Omega)$ such that $\mu_\rho \gammato \mu$ (for a subsequence that we have relabelled).
	Equivalently,
	the derivatives satisfy
	$S_\rho'(f_\rho) \WOT L_{\mu}$.
\end{prop}

The key observation that allows us to prove the next theorem, which is the main result, is the following.
\cref{lem:existenceOfmuStar} implies, under the stated assumptions,
that there is a subsequence (which we shall relabel)
such that for every $d \in H^{-1}(\Omega)$,
$\alpha_\rho := S_\rho'(f_\rho) d \weaklyto L_{\mu}d =: \alpha$.
We emphasise that the subsequence is independent of $d$.
Thus, all of the results of \cref{sec:fixed_d} can be applied to obtain information on $\alpha$
and, consequently,
on the limiting operator $L_{\mu}$.

We come now to our main result, which concatenates the above results and characterises the limiting elements obtained in \cref{prop:existenceQGeneral}. We will use the fact that the result of \cref{lem:orthogonality} is equivalent \cite[Lemma~3.7]{HarderWachsmuth} to 
\begin{align}
\alpha &= 0 \quad \text{q.e. on $A_s$}.\label{eq:harderwachsmuth_characterisation}
\end{align}
\begin{theorem}\label{thm:main_result}
{Let \cref{ass:for_Bsw,ass:1} hold.}
Let $f_\rho \to f$ in $H^{-1}(\Omega)$ with $S(f) \in C_0(\Omega)$ be given.
We further assume that $S_\rho$ is Gâteaux differentiable at $f_\rho$ for all $\rho > 0$.
Then there exists a map $\Qf \in  \partial_B^{sw}S(f)$ such that, for a subsequence (that we relabel),
\begin{equation*}
S_\rho'(f_\rho) \WOT \Qf.
\end{equation*}
\end{theorem}
\begin{proof}
Recalling the characterisation
\[\partial_B^{sw} S(f) = \{ L_\mu : \mu \in \M_0(\Omega),\; \mu(I) = 0 ,\; \mu= \infty \text{ on } A_s\},\]
we will show that the operator $L_{\mu}$ from \cref{lem:existenceOfmuStar} belongs to the set on the right-hand side.

We choose a nonnegative $v \in H_0^1(\Omega)$ such that $\set{v > 0} = I$.
Further, set $d := -\Delta v$ and $\alpha := L_{\mu} d$, i.e.,
\begin{equation}
-\Delta \alpha + \mu \alpha = -\Delta v\label{eq:alphamuandv}    
\end{equation}
(recall that the notation $L_{\mu}$ was defined just below \eqref{eq:defn_Lmu}).
Testing with $\alpha$ we get $\norm{\nabla\alpha}{L^2} \le \norm{\nabla v}{L^2}$. Consequently,
    \[\langle -\Delta (\alpha - v), \alpha - v \rangle
    \le 2 \langle -\Delta v, v \rangle + 2 \langle \Delta \alpha, v \rangle = 2 \langle d, v \rangle + 2 \langle \Delta \alpha, v \rangle = 0\]
with the final equality by \cref{prop:multiplerConditionsNew_purple}. Thus, $\alpha = v$. Testing the equation {\eqref{eq:alphamuandv} by $v$}, we get
\[    \int_\Omega v^2 \d\mu = 0\]
and this means $\mu(I) = 0$.

Now, we choose $d := 1$ and $\alpha := L_{\mu} d$.
Note that this $\alpha$ coincides with $w_\mu:= L_{\mu}(1)$ defined below Theorem 3.8 of \cite{RW}.
From the characterisation \eqref{eq:harderwachsmuth_characterisation} of the result of \cref{lem:orthogonality}, we get $\alpha = w_\mu = 0$ q.e. on $A_s$ and by \cite[Lemma 5.2]{RW} we get $\mu = +\infty$ on $A_s$.
\end{proof}
In the nonsmooth case, \cref{thm:main_result} and \cref{lem:density_result_general} imply that there exists a dense subset $F \subset H^{-1}(\Omega)$ such that if $f \in F$ and $S(f) \in C_0(\Omega)$, there exists a map $\Qf \in  \partial_B^{sw}S(f)$ such that, for a subsequence (that we relabel),
\begin{align*}
S_n'(f) &\WOT \Qf.
\end{align*}

\section{Optimal control of the obstacle problem}
We apply our findings to the optimal control of the obstacle problem.
Let $F_{ad} \subset L^2(\Omega)$ be a nonempty, closed and convex set satisfying
\[S(f) \in C_0(\Omega) \qquad \forall f \in F_{ad}\]
(see \cref{rem:cts_solution}) and let $J\colon H^1_0(\Omega) \times L^2(\Omega) \to \mathbb{R}$ be a given objective function satisfying
\begin{enumerate}[label=(\roman*)]\itemsep=0em
\item $J$ is continuously Fr\'echet differentiable with partial derivatives $J_y$ and $J_f$,
\item If $y_n \to y$ in $H^1_0(\Omega)$ and $f_n \weaklyto f$ in $L^2(\Omega)$, then
\[J(y,f) \leq \liminf_{n \to \infty} J(y_n,f_n).\]
\end{enumerate}
An example of $J$ satisfying the above conditions is
\[J(y,f) := \frac12 \norm{y-y_d}{L^2(\Omega)}^2 + \frac{\nu}{2}\norm{f}{L^2(\Omega)}^2,\]
where $y_d \in L^2(\Omega)$ is a given desired state and $\nu \geq 0$ is a constant. 

Consider the optimal control problem
\begin{equation}
\min_{f \in F_{ad}} J(y,f) \quad\text{such that}\quad y=S(f).\label{eq:OC_problem}
\end{equation}
In the next result, we will derive a first-order optimality condition for this control problem.
This rigorously shows the satisfaction of \cite[(19)]{RW}, which was derived only formally there.
\begin{theorem}\label{thm:oc_result}
{Let \cref{ass:for_Bsw} hold.}
For any local minimiser $(\bar y,\bar f) \in H^1_0(\Omega) \times F_{ad}$ of \eqref{eq:OC_problem},
there exists $\Qf \in \partial_B^{sw} S(\bar f)$
such that
\[0 \in \Qf^* J_y(\bar y,\bar f) + J_f(\bar y,\bar f) + N_{F_{ad}}(\bar f)\]
is satisfied where $N_{F_{ad}}(\bar f)$ is the normal cone to $F_{ad}$ at $\bar f$.
\end{theorem}
\begin{proof}
	We denote by $\varepsilon > 0$ the radius of optimality,
	i.e.,
	\begin{equation*}
		J(S(f),f) \ge J(\bar y, \bar f)
		\qquad
		\forall f \in F_{ad} \cap B_{\varepsilon}(\bar f),
	\end{equation*}
	where $B_\varepsilon(\bar f)$ is a closed ball in $L^2(\Omega)$.

	We regularise the problem \eqref{eq:OC_problem} and consider 
	\[
		\min_{f \in F_{ad} \cap B_\varepsilon(\bar f)} J(S_\rho(f),f) + \frac12 \norm{f - \bar f}{L^2(\Omega)}^2
		,
	\]
	where we take $S_\rho$ to be $S_\rho^{\mathrm{sm}}$ or $S_\rho^{\mathrm{sc}}$.
	Denote by $f_\rho$ a global minimiser of this problem and set $y_\rho := S_\rho(f_\rho)$.

	Using standard arguments,
	one can show $f_\rho \to \bar f$ in $L^2(\Omega)$.
	Consequently, the constraint $f_\rho \in B_\varepsilon(\bar f)$
	is not binding for small enough $\rho$.
	By the standard minimisation principle, we get 
	\[
		0 \in \bar J'(f_\rho) + (f_\rho - \bar f) + N_{F_{ad}}(f_\rho)
	\]
	where $\bar J(f) := J(S_\rho(f), f)$ defines the reduced functional. 
	Using the chain rule and the fact that $J$ is Fr\'echet, we can write this as
	\[
		0
		\in
		S_\rho'(f_\rho)^* J_y(y_\rho, f_\rho) + J_f(y_\rho, f_\rho) + (f_\rho - \bar f) + N_{F_{ad}}(f_\rho)
		\quad\text{in } L^2(\Omega).
	\]
	Note that $J_y(y_\rho, f_\rho) \in H^{-1}(\Omega)$
	and $S_\rho'(f_\rho)^* \colon H^{-1}(\Omega) \to H_0^1(\Omega) \subset L^2(\Omega)$.

	It remains to pass to the limit with this optimality condition.
	From \cref{prop:convergence_to_VI_soln}
	we get $y_\rho \to \bar y$ in $H_0^1(\Omega)$.
	Further,
	\cref{thm:main_result} enables us to select a subsequence
	(which we relabel)
	such that
	$S_\rho'(f_\rho) \WOT \Qf$ for some $\Qf \in  \partial_B^{sw}S(\bar f)$.
	Further, we note that $S_\rho'(f_\rho)$ and $\Qf$ are self-adjoint,
	which yields
	$S_\rho'(f_\rho)^* \WOT \Qf^*$
	Together with the product rule
	\cite[Lemma 2.9 (ii)]{RW}
	we get
	\[
		S_\rho'(f_\rho)^* J_y(y_\rho, f_\rho) + J_f(y_\rho, f_\rho) + (f_\rho - \bar f)
		\weaklyto
		\Qf^* J_y(\bar y, \bar f) + J_f(\bar y, \bar f)
	\]
	in $H_0^1(\Omega)$
	and, consequently,
	strongly in $L^2(\Omega)$.
	Since the graph of the normal cone map is closed,
	this implies
	\[
		0 \in \Qf^* J_y(\bar y, \bar f) + J_f(\bar y, \bar f) + N_{F_{ad}}(\bar f)
	\]
	as claimed.
\end{proof}
Thanks to this result, \cite[Lemma 7.2]{RW} implies the existence of 
\[p \in H^1_0(\Omega \setminus A_s),\qquad \nu \in H^{-1}(\Omega),\qquad \lambda \in N_{F_{ad}}(\bar f),\]
such that 
\begin{align*}
p + J_f(\bar y,\bar f) + \lambda &= 0, &
\langle \nu, v \rangle &= 0 \quad \forall v \in H^1_0(\Omega \setminus A),\\
J_y(\bar y,\bar f) + \Delta p - \nu &=0, &
\langle \nu, p\varphi \rangle &\geq 0 \quad \forall \varphi \in W^{1,\infty}(\Omega)^+,
\end{align*}
which is a necessary condition satisfied by every local minimiser as shown in \cite{SchielaWachsmuth}.
Here, $A_s$ and $A$ denote the strictly active set and the active set
associated with $(\bar y, \bar f)$, respectively. {In summary, we have derived first-order conditions with a different method of proof to what is available in the literature, by simply applying the limiting subdifferential theory we have developed without needing to study for example properties of adjoints.}

\printbibliography

\end{document}